\documentclass[10pt,reqno]{amsart}

\def\squad{\hskip0.75em\relax}

\usepackage{setspace}
\setdisplayskipstretch{1.5}
\usepackage[T1]{fontenc}

\usepackage{mathtools}
\usepackage{amsmath,amsfonts,amsbsy,amsgen,amscd,mathrsfs,amssymb,amsthm}
\usepackage{enumerate}
\usepackage{bm}
\usepackage{calc}

\usepackage[usenames,dvipsnames]{xcolor}
\usepackage[colorlinks=true,citecolor=blue,linkcolor=blue]{hyperref}

\usepackage{tikz}
\usepackage[font=small]{caption}

\newtheorem{thm}{Theorem}[section]
\newtheorem{lem}[thm]{Lemma}

\newtheorem{prop}[thm]{Proposition}
\newtheorem{cor}[thm]{Corollary}

\theoremstyle{definition}

\newtheorem*{rem*}{Remark}

\numberwithin{equation}{section} 
\numberwithin{figure}{section}
\numberwithin{table}{section}

\makeatletter

\newcommand{\id}{\mathbf{1}}
\newcommand{\spc}{\mathrm{sp}}

\begin{document}

\title
[Computing extreme singular values of free operators]
{\vspace*{-.6cm}
Computing extreme singular values \\ of free operators}

\author[Parmaks\i z]{Emre Parmaks\i z}
\address{Department of Mathematics, Massachusetts Institute of 
Technology, Cambridge, MA 02139, USA}
\email{yep@mit.edu}

\author[Van Handel]{Ramon van Handel}
\address{Department of Mathematics, Princeton University, Princeton, NJ 
08544, USA}
\email{rvan@math.princeton.edu}

\begin{abstract}
A recent development in random matrix theory, the intrinsic freeness 
principle, establishes that the spectrum of very general random matrices 
behaves as that of an associated free operator. This reduces the study of 
such random matrices to the deterministic problem of computing spectral 
statistics of the free operator. In the self-adjoint case, the spectral 
edges of the free operator can be computed exactly by means of a 
variational formula due to Lehner. In this note, we provide variational 
formulas for the largest and smallest singular values in the 
non-self-adjoint case.
\end{abstract}


\maketitle

\raggedbottom

\section{Introduction}
\label{sec:intro}

Let $X$ be a $d\times d$ random matrix with jointly Gaussian entries.
Without loss of generality, such a matrix can always be represented as
$$
	X = a_0 + \sum_{i=1}^n a_i g_i,
$$
where $g_1,\ldots,g_n$ are i.i.d.\ standard Gaussian variables
and $a_0,\ldots,a_n\in\mathbb{C}^{d\times d}$ are deterministic 
matrices.
Since no structural assumptions are made, the entries of such 
matrices can be highly nonhomogeneous and dependent.

Nonetheless, recent
advances in random marix theory make it possible to understand the 
spectrum of such matrices under surprisingly minimal assumptions.
To this end, define the deterministic operator
\begin{equation}
\label{eq:x}
	x = a_0\otimes\id + 
	\sum_{i=1}^n a_i\otimes s_i,
\end{equation}
where $s_1,\ldots,s_n$ is a free semicircular family 
(cf.\ \cite{NS06} or \cite[\S 4.1]{BBV23}). The \emph{intrinsic 
freeness principle} \cite{BBV23,BCSV24} establishes that under 
mild conditions, the spectral statistics of $X$ and $x$ nearly 
coincide: in the self-adjoint case, the spectra of $X$ and 
$x$ are close in the Hausdorff distance, while in the general case the 
same is true for the singular value spectrum. The universality principles
of \cite{BV24} further extend these conclusions to a large family of 
non-Gaussian random matrices.

The power of these results lies in the fact that they reduce the study of 
complicated random matrices to the deterministic problem of computing the 
spectrum of a free operator, which is accessible using free probability 
theory. In particular, when $x$ is self-adjoint, the following variational 
principle due to Lehner \cite{Leh99} (see \S\ref{sec:lehner}) 
provides an explicit 
formula for the spectral edges of $x$. In the following, we write 
$\lambda_{\rm max}(x)=\sup\spc(x)$ and $\lambda_{\rm 
min}(x)=\inf\spc(x)= -\lambda_{\rm max}(-x)$.

\begin{thm}[Lehner]
\label{thm:lehner}
Let $x$ be as in \eqref{eq:x} with 
$a_0,\ldots,a_n\in \mathbb{C}^{d\times d}_{\rm s.a.}$. Then
$$
	\lambda_{\rm max}(x) =
	\inf_{z>0} \, \lambda_{\rm max}\Bigg(
	a_0 + z^{-1} + \sum_{i=1}^n a_i z a_i
	\Bigg),
$$
where $z\in \mathbb{C}^{d\times d}_{\rm s.a.}$. Moreover, the infimum can 
be restricted to those $z$ such that the matrix in $\lambda_{\rm 
max}({\cdots})$ is a multiple of the identity.
\end{thm}

Lehner's formula is extremely useful in applications that require a 
precise understanding of the spectral edges of nonhomogeneous and 
dependent random matrices. For example, several such applications can be 
found in \cite{BCSV24}.

When $x$ is not self-adjoint, one is typically interested in the largest 
and smallest singular values $\mathrm{s}_{\rm 
max}(x)=\sup\spc((xx^*)^{1/2})$ and $\mathrm{s}_{\rm min}(x) = 
\inf\spc((xx^*)^{1/2})$, respectively. More generally, in problems related 
to sample covariance matrices, one is interested in the spectral edges of 
the operator $xx^* + b\otimes\id$ with $b\in\mathbb{C}^{d\times d}_{\rm 
s.a.}$. We emphasize at the outset that the computation of these 
quantities is in principle fully settled in \cite[\S 5.2]{Leh99}: Lehner 
provides a recipe for computing the operator norm of \emph{any} 
noncommutative quadratic polynomial of a free semicircular family with 
matrix coefficients, of which the above quantities are special cases. 
However, this general recipe gives rise to complicated and in some cases 
inexplicit variational formulas that prove to be difficult to use in 
concrete situations.

The aim of this note is to obtain much simpler explicit formulas, in the 
spirit of Theorem~\ref{thm:lehner}, for the spectral edges of the operator 
$xx^* + b\otimes\id$ in the case that $x$ is centered (that is, $a_0=0$ in 
\eqref{eq:x}). This setting may be viewed as a matrix-valued analogue of 
the free Poisson distribution, as will be 
explained in \S\ref{sec:cauchy} below. Even though this is a special 
case of the more general problem considered by Lehner, it is one that 
arises frequently in applications, and tractable formulas for the spectral 
edges are essential for the analysis of concrete models.

\subsection{Main result}

The following setting will be considered throughout the 
remainder of this note. Fix $d,m,n\in\mathbb{N}$, deterministic
matrices $a_1,\ldots,a_n\in\mathbb{C}^{d\times m}$, and a self-adjoint 
deterministic matrix $b\in\mathbb{C}^{d\times d}_{\rm s.a.}$.
Define the operator
\begin{equation}
\label{eq:mainx}
	x = \sum_{i=1}^n a_i\otimes s_i,
\end{equation}
where $s_1,\ldots,s_n$ is a free semicircular family.

\begin{thm}
\label{thm:main}
For $x$ as in \eqref{eq:mainx}, we have
$$
	\lambda_{\rm max}(xx^* + b\otimes\id) =
	\inf_{\substack{z>0 \\ \sum_j \! a_j^*za_j<\id}}
	\lambda_{\rm max}\Bigg(
	b + z^{-1} +
	\sum_{i=1}^n a_i \Bigg(
	\id - \sum_{j=1}^n a_j^*z a_j
	\Bigg)^{-1} a_i^*
	\Bigg)
$$
and
$$
	\lambda_{\rm min}(xx^* + b\otimes\id) =
	\sup_{z<0}\,
	\lambda_{\rm min}\Bigg(
	b + z^{-1} +
	\sum_{i=1}^n a_i \Bigg(
	\id - \sum_{j=1}^n a_j^*z a_j
	\Bigg)^{-1} a_i^*
	\Bigg),
$$
where 
$z\in \mathbb{C}^{d\times d}_{\rm s.a.}$.
Moreover, the infimum \emph{(}supremum\emph{)} can be restricted to
those $z$ such that the matrix in
$\lambda_{\rm max}({\cdots})$ \emph{(}$\lambda_{\rm 
min}({\cdots})$\emph{)} is a multiple of the identity.
\end{thm}

As an immediate corollary, we obtain variational formulas for the
largest and smallest singular values $\mathrm{s}_{\rm max}(x)$ and
$\mathrm{s}_{\rm min}(x)$ of $x$.

\begin{cor}
\label{cor:sing}
For $x$ as in \eqref{eq:mainx}, we have
$$
	\mathrm{s}_{\rm max}(x)^2 = 
	\inf_{\substack{z>0 \\ \sum_j \! a_j^*za_j<\id}}
	\lambda_{\rm max}\Bigg(
	z^{-1} +
	\sum_{i=1}^n a_i \Bigg(
	\id - \sum_{j=1}^n a_j^*z a_j
	\Bigg)^{-1} a_i^*
	\Bigg)
$$
and
$$
	\mathrm{s}_{\rm min}(x)^2 = 
	\sup_{z<0}\,
	\lambda_{\rm min}\Bigg(
	z^{-1} +
	\sum_{i=1}^n a_i \Bigg(
	\id - \sum_{j=1}^n a_j^*z a_j
	\Bigg)^{-1} a_i^*
	\Bigg).
$$
\end{cor}

\begin{proof}
Apply Theorem \ref{thm:main} with $b=0$.
\end{proof}

\subsection{A reduction principle}

We now state a reduction principle that facilitates the application of 
Theorem~\ref{thm:main} to models with symmetries. An analogous 
result in the setting of Theorem~\ref{thm:lehner} appears in
\cite[Lemma~7.1]{BCSV24}.

\begin{lem}
\label{lem:reduction}
Let $\mathcal{A}$ be a $*$-subalgebra of
$\mathbb{C}^{m\times m}$ and $\mathcal{B}$ be a $*$-subalgebra
of $\mathbb{C}^{d\times d}$ such that the following conditions hold:
$$
	b\in\mathcal{B},\qquad
	\sum_{i=1}^n a_iya_i^*\in\mathcal{B}
	\squad\text{for all}\squad y\in\mathcal{A},\qquad
	\sum_{j=1}^n a_j^*za_j\in\mathcal{A}
	\squad\text{for all}\squad z\in\mathcal{B}.
$$
Then for $x$ as in \eqref{eq:mainx}, all the conclusions of
Theorem \ref{thm:main} remain valid if the infimum 
\emph{(}supremum\emph{)} is taken only over $z\in\mathcal{B}$.
\end{lem}

Let us illustrate this reduction principle in a special
case that is important in applications. Let $X$ be a $d\times m$ random 
matrix with independent centered Gaussian entries $X_{ij}\sim 
N(0,\sigma_{ij}^2)$. The associated free model $x$ is given by
\begin{equation}
\label{eq:indep}
	x = \sum_{i=1}^d \sum_{j=1}^m
	\sigma_{ij}\,e_ie_j^*\otimes s_{ij}
\end{equation}
where $(s_{ij})_{i\in[d],j\in[m]}$ are free semicircular variables
and $e_1,e_2,\ldots$ denote the standard basis vectors in 
$\mathbb{C}^d$ or $\mathbb{C}^m$. Then we have the following.

\begin{cor}
\label{cor:indep}
Let $x$ have independent entries as in \eqref{eq:indep}, and let
$b$ be a diagonal matrix with diagonal entries $b_1,\ldots,b_d$.
Then we have
$$
	\lambda_{\rm max}(xx^* + b\otimes\id) =
	\inf_{\substack{\min_j v_j>0 \\ 
	\max_j \sum_k \sigma_{kj}^2 v_k < 1}}
	\max_{i\in[d]}
	\Bigg(
	b_i + \frac{1}{v_i} + 
	\sum_{j\in[m]} \frac{\sigma_{ij}^2}{1 - \sum_{k\in[d]} \sigma_{kj}^2v_k}
	\Bigg)
$$
and
$$
	\lambda_{\rm min}(xx^* + b\otimes\id) =
	\sup_{\max_j v_j<0}\,
	\min_{i\in[d]}
	\Bigg(
	b_i + \frac{1}{v_i} + 
	\sum_{j\in[m]} \frac{\sigma_{ij}^2}{1 - \sum_{k\in[d]} \sigma_{kj}^2v_k}
	\Bigg),
$$
where $v\in\mathbb{R}^d$. Moreover, the infimum
\emph{(}supremum\emph{)} can be restricted to those $v$ such that the
vector in the maximum \emph{(}minimum\emph{)} is a multiple of the
constant vector $1$.
\end{cor}

\begin{proof}
Let $\mathcal{A}$ and $\mathcal{B}$ be the $*$-algebras of diagonal 
matrices
in $\mathbb{C}^{m\times m}$ and $\mathbb{C}^{d\times d}$, respectively.
Then it is readily verified that the assumptions of Lemma 
\ref{lem:reduction} are satisfied, and the conclusion follows by
restricting the variational principles of Theorem~\ref{thm:main} to
diagonal matrices $z$ with diagonal entries $v_1,\ldots,v_d$.
\end{proof}

An application of Corollary \ref{cor:indep} to phase transitions of 
nonhomogeneous sample covariance matrices is developed in \cite[\S 3.6 and 
\S 8.5]{BCSV24}.

\subsection{Organization of this note}

In \S\ref{sec:prelim}, we 
recall some results of matrix algebra and the formulas of Lehner that will 
be used in the proofs. The first variational formula of 
Theorem~\ref{thm:main} is proved in \S\ref{sec:upper}, and the second 
variational formula is proved in \S\ref{sec:lower}. The last 
statement of Theorem~\ref{thm:main} and Lemma~\ref{lem:reduction} are 
proved in \S\ref{sec:cauchy}.

\section{Preliminaries}
\label{sec:prelim}

\subsection{Matrix algebra}

The following two theorems of matrix algebra
will play an important role in our proofs. The first is a classical
result on Schur complements.

\begin{lem}[Schur complements]
\label{lem:schur}
Let 
$A\in\mathbb{C}^{d_1\times d_1}_{\rm s.a.}$, 
$B\in\mathbb{C}^{d_1\times d_2}$, 
$D\in\mathbb{C}^{d_2\times d_2}_{\rm s.a.}$, and
$$
	M = \begin{bmatrix}
	A & B \\
	B^* & D
	\end{bmatrix},
$$
and assume that $D$ is invertible. Define the Schur complement
$$
	M/D = A-BD^{-1}B^*.
$$
Then 
$$
	M>0\qquad \text{if and only if}\qquad
	D>0\squad\text{and}\squad M/D>0.
$$
Moreover, if $M$ and $M/D$ are invertible, we have
$$
	M^{-1} =
	\begin{bmatrix}
	(M/D)^{-1} & -(M/D)^{-1}BD^{-1} \\
	-D^{-1}B^*(M/D)^{-1} & D^{-1} + D^{-1}B^*(M/D)^{-1}BD^{-1}
	\end{bmatrix}.
$$
\end{lem}

\smallskip

The first statement can be found in \cite[Theorem 1.12]{Zha05} and
the second statement can be found in \cite[Theorem 1.2]{Zha05}.

The second result is a simple matrix inversion identity that is closely
related to the Schur complement. Its correctness is readily verified
algebraically.

\begin{lem}[Matrix inversion lemma]
\label{lem:woodbury}
Let $B\in\mathbb{C}^{d_1\times d_2}$ and
$D\in\mathbb{C}^{d_2\times d_2}_{\rm s.a.}$. Then
$$
	(\id - BD^{-1}B^*)^{-1} =
	\id + B(D-B^*B)^{-1}B^*,
$$
provided that all the inverses in this identity exist.
\end{lem}

\subsection{Lehner's formulas}
\label{sec:lehner}

Theorem~\ref{thm:lehner} is essentially \cite[Corollary~1.5]{Leh99}, but 
it is phrased there in a slightly different manner. For completeness, we 
spell out the straightforward translation to the form given here.

\begin{proof}[Proof of Theorem \ref{thm:lehner}]
Choose $c>0$ so that $a_0+c\id\ge 0$ and
$x+c\id\ge 0$ (this is possible since $x$ is a bounded operator).
By \cite[Corollary~1.5]{Leh99}, we have
$$
	\|x+c\id\| = \inf_{z>0}\Bigg\|
	a_0+c\id + z^{-1} + \sum_{i=1}^n a_i z a_i
	\Bigg\|,
$$
where the infimum may also be restricted to those $z$ such that
the matrix inside the norm is a multiple of the identity. The conclusion
follows by applying to both sides of this identity the elementary fact,
that if $y$ is self-adjoint and $y+c\id\ge 0$, then
$\|y+c\id\| = \lambda_{\rm max}(y+c\id) = \lambda_{\rm max}(y)+c$.
\end{proof}

For the proof of the second formula of Theorem \ref{thm:main}, 
we will need a more general form Theorem \ref{thm:lehner}.
To state it, we must first recall a standard construction of a free 
semicircular family on the free Fock space, 
cf.\ \cite[pp.\ 102--108]{NS06}. Let
$$
	\mathcal{F}(\mathbb{C}^n) :=
	\bigoplus_{k=0}^\infty (\mathbb{C}^n)^{\otimes k}
$$
be the free Fock space over $\mathbb{C}^n$, where by convention
$(\mathbb{C}^n)^{\otimes 0}$ denotes the one-dimensional Hilbert space
spanned by the vacuum state $\Omega$. For $i=1,\ldots,n$, 
we define the creation operator $l_i$ on $\mathcal{F}(\mathbb{C}^n)$ by
$$
	l_i(x_1\otimes\cdots\otimes x_n) := e_i\otimes x_1\otimes\cdots\otimes x_n,
$$
where $e_1,\ldots,e_n$ is the coordinate basis in $\mathbb{C}^n$,
and we define
$$
	s_i := l_i + l_i^*.
$$
Then $s_1,\ldots,s_n$ is a free semicircular family. Moreover, we have
the identities
\begin{equation}
\label{eq:fockid}
	l_i^*l_j = 1_{i=j}\id,\qquad\quad
	\sum_{i=1}^n l_il_i^* = \id - p_\Omega,
\end{equation}
where $p_\Omega$ denotes the orthogonal projection on
$(\mathbb{C}^n)^{\otimes 0}$.

We now state a generalization of Theorem \ref{thm:lehner} to self-adjoint 
linear combinations of the operators $l_i,p_\Omega$ with matrix 
coefficients.

\begin{thm}[Lehner]
\label{thm:lehnerfock}
Let $a_0,b\in \mathbb{C}^{d\times d}_{\rm s.a.}$ and
$a_1,\ldots,a_n\in \mathbb{C}^{d\times d}$. Then
$$
	\lambda_{\rm max}\Bigg(
	a_0\otimes\id + b\otimes p_\Omega +
	\sum_{i=1}^n \big( a_i\otimes l_i+a_i^*\otimes l_i^*\big)
	\Bigg) = 
	\inf_{\substack{z>0 \\ z^{-1}>b}}
	\lambda_{\rm max}\Bigg(
	a_0 + z^{-1} + \sum_{i=1}^n a_i^* z a_i
	\Bigg),
$$
where $z\in \mathbb{C}^{d\times d}_{\rm s.a.}$. Moreover, the infimum can 
be restricted to those $z$ such that the matrix in $\lambda_{\rm 
max}({\cdots})$ is a multiple of the identity.
\end{thm}

\begin{proof}
This follows from \cite[Theorem 1.8]{Leh99} as in the above proof of 
Theorem 
\ref{thm:lehner} and with the substitution $z+b \leftarrow z^{-1}$.
\end{proof}

We finally recall a classical dilation result that will be used in 
conjunction with Lehner's formulas; cf.\ \cite[Lemma 4.9]{BBV23}.

\begin{lem}
\label{lem:dilation}
Let $y:H_1\to H_2$ be a bounded linear operator between Hilbert spaces
$H_1$ and $H_2$, and let $\tilde y$ be the bounded operator on
$H_1\oplus H_2$ defined by
$$
	\tilde y = \begin{bmatrix}
	0 & y^* \\ y & 0
	\end{bmatrix}.
$$
Then $\spc(\tilde y) \cup \{0\} =
\spc((yy^*)^{1/2}) \cup {-\spc((yy^*)^{1/2})} \cup \{0\}$.
\end{lem}

\section{The upper edge}
\label{sec:upper}

The aim of this section is to prove the first variational formula of 
Theorem~\ref{thm:main}. This formula will be deduced from 
Theorem~\ref{thm:lehner} by a linearization argument. As we will see, 
however, the reduction to the simple form given in Theorem~\ref{thm:main} 
requires a careful analysis of the resulting optimization problem.

Let us note at the outset that as $\lambda_{\rm max}(y+c\id)= \lambda_{\rm 
max}(y)+c$ for any $c\in\mathbb{R}$, the variational formula in 
Theorem~\ref{thm:main} is unchanged if we replace $b$ by $b+c\id$. By 
choosing $c$ large enough, we may assume without loss of generality that 
$b>0$. This assumption will henceforth be made throughout this section.
In particular, this enables us to write $b=a_0a_0^*$ for some 
$a_0\in\mathbb{C}^{d\times d}$.

We begin with a reduction to the setting of
Theorem \ref{thm:lehner}.

\begin{lem}
\label{lem:lin}
Let $r=m+2d$ and
$$
	\tilde x = \begin{bmatrix}
	0 & 0 & a_0^*\otimes\id \\
	0 & 0 & x^* \\
	a_0\otimes\id & x & 0
	\end{bmatrix} =
	\tilde a_0 + \sum_{i=1}^n \tilde a_i\otimes s_i,
$$
where for $i=1,\ldots,n$ we define the $r\times r$ matrices
$$
	\tilde a_0 = 
	\begin{bmatrix}
	0 & 0 & a_0^* \\
	0 & 0 & 0 \\
	a_0 & 0 & 0
	\end{bmatrix},\qquad\quad
	\tilde a_i = 
	\begin{bmatrix}
	0 & 0 & 0 \\
	0 & 0 & a_i^* \\
	0 & a_i & 0
	\end{bmatrix}.
$$
Then $\lambda_{\rm max}(\tilde x)
=\lambda_{\rm max}(xx^*+b\otimes\id)^{1/2}$.
\end{lem}

\begin{proof}
This follows directly from Lemma \ref{lem:dilation}.
\end{proof}

Theorem \ref{thm:lehner} immediately yields the following.

\begin{cor}
\label{cor:upperlehner}
Define 
$$
	f(\tilde z) = \tilde a_0 + \tilde z^{-1}
	+ \sum_{i=1}^n \tilde a_i \tilde z\tilde a_i.
$$
Then we have
$$
	\lambda_{\rm max}(xx^*+b\otimes\id)^{1/2} =
	\inf_{\tilde z>0} \lambda_{\rm max}(f(\tilde z)) = 
	\inf_{\lambda>0} \inf_{\substack{\tilde z>0\\ f(\tilde z)=\lambda\id}}
	\lambda_{\rm max}(f(\tilde z)),
$$
where the infimum is taken over $\tilde z\in\mathbb{C}^{r\times r}_{\rm 
s.a.}$.
\end{cor}

Next, we identify some consequences of
the identity $f(\tilde z)=\lambda\id$.

\begin{lem}
\label{lem:schurred}
Write $\tilde z\in\mathbb{C}^{r\times r}_{\rm s.a.}$ in the
block decomposition of Lemma \ref{lem:lin} as
$$
	\tilde z = 
	\begin{bmatrix}
		p & v_1 & v_2 \\
		v_1^* & q & w \\
		v_2^* & w^* & z
	\end{bmatrix}.
$$
Let $\lambda>0$, and suppose that $\tilde z>0$ and $f(\tilde z)=\lambda\id$.
Then
$$
	z>0,\qquad
	\sum_{j=1}^n a_j^*za_j < \lambda\id,\qquad
	\lambda^{-1}b + z^{-1} 
	+\sum_{i=1}^n a_i\Bigg(\lambda\id -
	\sum_{j=1}^n a_j^*za_j\Bigg)^{-1}a_i^*
	\le \lambda\id.
$$
\end{lem}

\begin{proof}
We begin by noting that
$$
	f(\tilde z) = 
	\begin{bmatrix}
	0 & 0 & a_0^* \\
	0 & 0 & 0 \\
	a_0 & 0 & 0
	\end{bmatrix}
	+ \tilde z^{-1} +
	\sum_{i=1}^n
	\begin{bmatrix}
	0 & 0 & 0 \\
	0 & a_i^*za_i & a_i^*w^*a_i^* \\
	0 & a_iwa_i & a_iqa_i^*
	\end{bmatrix}.
$$
To compute $\tilde z^{-1}$, we express
$$
	\tilde z = \begin{bmatrix}
	p & V \\ V^* & Z
	\end{bmatrix},
	\qquad
	Z = \begin{bmatrix}
	 q & w \\ w^* & z \end{bmatrix},\qquad
	V = \begin{bmatrix} v_1 & v_2 \end{bmatrix}
$$
and apply Lemma \ref{lem:schur}. Using that $f(\tilde z)=\lambda\id$ 
yields
\begin{align*}
&	(\tilde z/Z)^{-1} = \lambda\id,
\\
&
	\begin{bmatrix} 0 & a_0^*\end{bmatrix} 
	-(\tilde z/Z)^{-1}VZ^{-1} = 0,
\\
&
	Z^{-1} + Z^{-1}V^*(\tilde z/Z)^{-1}VZ^{-1} +
	\sum_{i=1}^n
	\begin{bmatrix} a_i^*za_i & a_i^*w^*a_i^* \\
	a_iwa_i & a_iqa_i^* \end{bmatrix}
	= \lambda\id.
\end{align*}
Plugging the first two equations into the last equation yields
$$
	Z^{-1} + \frac{1}{\lambda}
	\begin{bmatrix} 0 & 0 \\ 0 & b
	\end{bmatrix}
	+
	\sum_{i=1}^n
	\begin{bmatrix} a_i^*za_i & a_i^*w^*a_i^* \\
	a_iwa_i & a_iqa_i^* \end{bmatrix}
	= \lambda\id,
$$
where we used that $b=a_0a_0^*$. We note that the above equations are all 
well defined, as $\tilde z>0$ ensures that $Z>0$ and $\tilde z/Z>0$ by
Lemma \ref{lem:schur}.

Next, we compute $Z^{-1}$ using Lemma \ref{lem:schur} and plug the 
resulting expression into the last equation display. Then we obtain
\begin{align}
\label{eq:firstsch}
&	(Z/z)^{-1} + \sum_{i=1}^n a_i^*za_i = \lambda\id,
\\
\label{eq:secondsch}
&	z^{-1} + z^{-1}w^*(Z/z)^{-1}wz^{-1} +
	\lambda^{-1}b +
	\sum_{i=1}^n a_iqa_i^* = \lambda\id.
\end{align}
Moreover, as $Z>0$, we have $z>0$ and $Z/z>0$. This
yields the desired inequality 
$\lambda\id - \sum_i a_i^*za_i = (Z/z)^{-1} > 0$ by \eqref{eq:firstsch}.
Finally, by \eqref{eq:firstsch} and the definition of $Z/z$,
$$
	q \ge 
	q - wz^{-1}w^* 
	= Z/z = 
	\Bigg(
	\lambda\id-
	\sum_{i=1}^n a_i^*za_i
	\Bigg)^{-1},
$$
and thus the last inequality in the statement follows from 
\eqref{eq:secondsch}.
\end{proof}

The idea of the proof is now to show that given any $\tilde z$ as in 
Lemma \ref{lem:schurred}, we can find a matrix with a smaller
value for the variational principle of Corollary \ref{cor:upperlehner}.

\begin{lem}
\label{lem:valred}
Let $\lambda>0$ and
$z\in\mathbb{C}^{d\times d}_{\rm s.a.}$ such that
$z>0$ and $\sum_j a_j^*za_j<\lambda\id$. Define
$$
	g(\lambda;z) =
	\begin{bmatrix}
	\lambda^{-1}\id + \lambda^{-2}a_0^*za_0
	& 0 & \lambda^{-1}a_0^*z \\
	0 & \big(\lambda\id - \sum_j a_j^*za_j\big)^{-1} & 0 \\
	\lambda^{-1}za_0 & 0 & z
	\end{bmatrix}.
$$
Then $g(\lambda;z)>0$ and 
$$
	f(g(\lambda;z)) =
	\begin{bmatrix}
	\lambda\id & 0 & 0 \\
	0 & \lambda\id  & 0 \\
	0 & 0 & 
	\lambda^{-1}b + z^{-1} +
	\sum_i a_i \big(\lambda\id - \sum_j a_j^*za_j\big)^{-1} a_i^*
	\end{bmatrix}.
$$
In particular, if $\tilde z$ is as in Lemma \ref{lem:schurred}, then
$f(g(\lambda;z))\le f(\tilde z)$.
\end{lem}

\begin{proof}
Express $\hat z=g(\lambda;z)$ in block form as
$$
	\hat z = \begin{bmatrix}
	\lambda^{-1}\id + \lambda^{-2}a_0^*za_0 & W \\ W^* & Y
	\end{bmatrix}
$$
with
$$
	Y = \begin{bmatrix}
	 \big(\lambda\id - \sum_j a_j^*za_j\big)^{-1} & 0 \\ 0 & z 
	\end{bmatrix},\qquad
	W = \begin{bmatrix} 0 & \lambda^{-1}a_0^*z \end{bmatrix}.
$$
Then $Y>0$ and $\hat z/Y = \lambda^{-1}\id>0$, so that $\hat z>0$
by Lemma \ref{lem:schur}.

We can further use Lemma \ref{lem:schur} to compute 
$$
	\hat z^{-1} =
	\begin{bmatrix}
	\lambda\id & 0 & -a_0^* \\
	0 & \lambda\id - \sum_j a_j^*za_j & 0 \\
	-a_0 & 0 & \lambda^{-1}b + z^{-1}
	\end{bmatrix}.
$$
The identity for $f(\hat z)$ now follows readily. Finally, if
$\tilde z$ is as in Lemma \ref{lem:schurred}, we obtain
$f(\hat z)\le \lambda\id = f(\tilde z)$ using the last
inequality of Lemma \ref{lem:schurred}.
\end{proof}

We can now prove the first part of Theorem \ref{thm:main}.

\begin{prop}
\label{prop:mainupper}
For $x$ as in \eqref{eq:mainx}, we have
$$
	\lambda_{\rm max}(xx^* + b\otimes\id) =
	\inf_{\substack{z>0 \\ \sum_j \! a_j^*za_j<\id}}
	\lambda_{\rm max}\Bigg(
	b + z^{-1} +
	\sum_{i=1}^n a_i \Bigg(
	\id - \sum_{j=1}^n a_j^*z a_j
	\Bigg)^{-1} a_i^*
	\Bigg).
$$
\end{prop}

\begin{proof}
Corollary \ref{cor:upperlehner} and Lemma \ref{lem:valred} yield
\begin{align*}
	\lambda_{\rm max}(xx^*+b\otimes\id)^{1/2} &=
	\inf_{\lambda>0} \inf_{\substack{\tilde z>0\\ f(\tilde z)=\lambda\id}}
	\lambda_{\rm max}(f(\tilde z))
\\
&	\ge
	\inf_{\lambda>0}
	\inf_{\substack{z>0 \\ \sum_j \! a_j^*za_j<\lambda\id}}
	\lambda_{\rm max}(f(g(\lambda;z))) 
\\
&	\ge
	\inf_{\tilde z>0} \lambda_{\rm max}(f(\tilde z))
	=
	\lambda_{\rm max}(xx^*+b\otimes\id)^{1/2}.
\end{align*}
Therefore, the expression for $f(g(\lambda;z))$ in
Lemma \ref{lem:valred} yields
\begin{multline*}
	\lambda_{\rm max}(xx^*+b\otimes\id)^{1/2}  =
	\inf_{\lambda>0}
	\inf_{\substack{z>0 \\ \sum_j \! a_j^*za_j<\lambda\id}}
	\lambda_{\rm max}(f(g(\lambda;z))) = \\
	\inf_{\lambda>0}
	\inf_{\substack{z>0 \\ \sum_j \! a_j^*za_j<\lambda\id}}
	\max\Bigg\{
	\lambda,~
	\lambda_{\rm max}\Bigg(
	\lambda^{-1}b + z^{-1} +
	\sum_i a_i \Bigg(\lambda\id - \sum_{j=1}^n 
	a_j^*za_j\Bigg)^{-1} a_i^*
	\Bigg)
	\Bigg\}.
\end{multline*}
The conclusion follows readily by replacing
$z\leftarrow \lambda z$ in the inner infimum and using that
$\inf_{\lambda>0}
\max\{\lambda,\lambda^{-1}c\}=c^{1/2}$ for $c\ge 0$.
\end{proof}

\section{The lower edge}
\label{sec:lower}

The aim of this section is to prove the second variational formula of 
Theorem \ref{thm:main}. The basic structure of the proof is similar to 
that of the first variational formula. However, the complication that 
arises in the present case is that we must rely on a more involved 
linearization argument that appears in \cite[\S 5.2]{Leh99}. 

To this end, define the operator
$$
	L = \begin{bmatrix}
	\id\otimes p_\Omega & \id\otimes l_1^* &
	\id\otimes l_1 & \cdots & \id\otimes l_n^* & \id\otimes l_n
	\end{bmatrix}
$$
in $\mathbb{C}^{1\times (2n+1)}\otimes
\mathbb{C}^{d\times d}\otimes B(\mathcal{F}(\mathbb{C}^n))$, and
define the matrix
$$
	T = 
	\begin{bmatrix}
	\frac{a_0a_0^*}{n+1} & 0 & \cdots & 0 \\
	0 &
	\frac{a_0a_0^*}{n+1}-a_1a_1^*
	& & -a_1a_n^* \\
	0 & -a_1a_1^*
	& & -a_1a_n^* \\
	\vdots &&\ddots & \vdots \\
	0 & -a_na_1^* && -a_na_n^* \\
	0 & -a_na_1^* &\cdots&
	\frac{a_0a_0^*}{n+1}-a_na_n^*
	\end{bmatrix}
$$
in $\mathbb{C}^{(2n+1)\times(2n+1)}\otimes \mathbb{C}^{d\times d}$,
where $a_0\in\mathbb{C}^{d\times d}$ will be chosen shortly. Then
\eqref{eq:fockid} yields
$$
	L(T\otimes\id)L^* = a_0a_0^*\otimes\id - xx^*.
$$
Now recall, as was explained at the beginning of \S\ref{sec:upper},
that the variational formula that we aim to prove is unchanged if
we replace $b$ by $b+c\id$ for any $c\in\mathbb{R}$. By choosing
$c$ to be sufficiently negative, we may assume without loss of 
generality that we can write $-b=a_0a_0^*>0$ for some
$a_0\in\mathbb{C}^{d\times d}$, and that moreover $T$ is positive
definite. This is assumed henceforth throughout this section.

Since $T$ is positive definite, it can be expressed as $T=RR^*$
where
$$
	R = 
	\begin{bmatrix}
	\frac{a_0}{\sqrt{n+1}} & 0 \\
	0 & B
	\end{bmatrix},
\qquad\qquad
	B = \begin{bmatrix}
	M_1 \\ N_1 \\ \vdots \\ M_n \\ N_n
	\end{bmatrix}
$$
for some matrix $B\in\mathbb{C}^{2nd\times 2nd}$ that we expressed
in block form with blocks $M_i,N_i\in\mathbb{C}^{d\times 2nd}$.
Then we obtain the following analogue of Corollary 
\ref{cor:upperlehner}.

\begin{lem}
\label{lem:lowerlehner}
Let $s=2(n+1)d$, define the $s\times s$ matrices
$$
	A_0 = 
	\frac{1}{\sqrt{n+1}}
	\begin{bmatrix}
	0 & 0 & a_0^* \\
	0 & 0 & 0 \\
	a_0 & 0 & 0
	\end{bmatrix},\qquad\quad
	A_i = \begin{bmatrix}
	0 & 0 & 0\\
	0 & 0 & M_i^* \\
	0 & N_i & 0
	\end{bmatrix}
$$
for $i=1,\ldots,n$, and define
$$
	F(Z) = 
	Z^{-1} + \sum_{i=1}^n A_i^* Z A_i.
$$
Then
$$
	\big({-\lambda_{\rm min}(xx^*+b\otimes\id)}\big)^{1/2} =
	\inf_{\substack{Z>0 \\ Z^{-1}>A_0}}
	\lambda_{\rm max}(F(Z)) =
	\inf_{\lambda>0}
	\inf_{\substack{Z>0 \\ Z^{-1}>A_0 \\ F(Z)=\lambda\id}}
	\lambda_{\rm max}(F(Z)).
$$
\end{lem}

\begin{proof}
Note that we can write
$$
	A_0\otimes p_\Omega + \sum_{i=1}^n
	\big( A_i\otimes l_i + A_i^*\otimes l_i^*\big)
	=
	\begin{bmatrix}
	0 & (R\otimes\id)^*L^* \\
	L(R\otimes\id) & 0 \end{bmatrix},
$$
so that Lemma \ref{lem:dilation} yields
\begin{multline*}
	\lambda_{\rm max}\Bigg(
	A_0\otimes p_\Omega + \sum_{i=1}^n
	\big( A_i\otimes l_i + A_i^*\otimes l_i^*\big)
	\Bigg) = \\
	\lambda_{\rm max}\big(
	L(T\otimes\id)L^*
	\big)^{1/2}
	=
	\big({-\lambda_{\rm min}(xx^*+b\otimes\id)}\big)^{1/2}.
\end{multline*}
The conclusion now follows from Theorem \ref{thm:lehnerfock}.
\end{proof}

We can now proceed in a similar manner as in \S\ref{sec:upper}.

\begin{lem}
\label{lem:lowerschur}
Write $Z\in\mathbb{C}^{s\times s}_{\rm s.a.}$ in the block decomposition
of Lemma \ref{lem:lowerlehner} as
$$
	Z = \begin{bmatrix}
	P & V_1 & V_2 \\
	V_1^* & Q & W \\
	V_2^* & W^* & z
	\end{bmatrix}.
$$
Let $\lambda>0$, and suppose that $Z>0$, $Z^{-1}>A_0$, and
$f(Z)=\lambda\id$. Then
$$
	z^{-1} > \frac{1}{\lambda} \frac{a_0a_0^*}{n+1},\qquad\qquad
	\sum_{j=1}^n N_j^*zN_j < \lambda\id,
$$
and
$$
	z^{-1} + 
	\sum_{i=1}^n M_i \Bigg(\lambda\id - \sum_{j=1}^n N_j^*zN_j
	\Bigg)^{-1}
	M_i^* \le \lambda\id.
$$
\end{lem}

\begin{proof}
We begin by noting that
$$
	F(Z) = Z^{-1} + 
	\sum_{i=1}^n
	\begin{bmatrix}
	0 & 0 & 0 \\
	0 & N_i^*zN_i & N_i^*W^*M_i^* \\
	0 & M_iWN_i & M_iQM_i^*
	\end{bmatrix}.
$$
To compute $Z^{-1}$, we express
$$
	Z = \begin{bmatrix}
	P & V \\ V^* & Z_0
	\end{bmatrix},\qquad
	Z_0 = \begin{bmatrix} Q & W \\  W^* & z
	\end{bmatrix},\qquad
	V = \begin{bmatrix} V_1 & V_2 \end{bmatrix}
$$
and apply Lemma \ref{lem:schur}. Using that $F(Z)=\lambda\id$
yields $V=0$, $P=\lambda^{-1}\id$, and
$$
	Z_0^{-1} +
	\sum_{i=1}^n
	\begin{bmatrix} 
	N_i^*zN_i & N_i^*W^*M_i^* \\
	M_iWN_i & M_iQM_i^*
	\end{bmatrix} =
	\lambda\id,
$$
where we note that $Z>0$ implies that $Z_0>0$.

Next, computing $Z_0^{-1}$ using Lemma \ref{lem:schur}, we obtain
\begin{align}
\label{eq:l1}
	&(Z_0/z)^{-1} + \sum_{j=1}^n N_j^*zN_j = \lambda\id,
\\
\label{eq:l2}
&	z^{-1} +
	z^{-1} W^*(Z_0/z)^{-1}Wz^{-1} +
	\sum_{i=1}^n M_iQM_i^* = \lambda\id.
\end{align}
with $z>0$ and $Z_0/z>0$. 
This yields the desired inequality
$\lambda\id - \sum_{j=1}^n N_j^*zN_j = (Z_0/z)^{-1}>0$ by
\eqref{eq:l1}. Moreover, by \eqref{eq:l1} and the
definition of $Z_0/z$,
$$
	Q \ge Q-Wz^{-1}W^* = Z_0/z =
	\Bigg(\lambda\id - \sum_{j=1}^n N_j^*zN_j\Bigg)^{-1},
$$
and thus the last inequality in the statement follows from \eqref{eq:l2}.

Finally, note that as $P=\lambda^{-1}\id$ and $V=0$, the condition
$Z^{-1}>A_0$ becomes
$$
	\begin{bmatrix}
	\lambda\id & C^* \\
	C & Z_0^{-1}
	\end{bmatrix} >
	0,
	\qquad\quad
	C = - \frac{1}{\sqrt{n+1}}\begin{bmatrix} 0 \\ a_0
	\end{bmatrix},
$$
so that Lemma \ref{lem:schur} implies that
$\lambda\id - C^*Z_0C>0$. Thus $\frac{a_0^*za_0}{n+1}<\lambda\id$.
As we assumed that $a_0a_0^*>0$, the matrix $a_0$ is invertible, and thus
taking the inverse of both sides and rearranging yields the desired 
inequality $z^{-1}>\frac{1}{\lambda}\frac{a_0a_0^*}{n+1}$.
\end{proof}

The above lemma gives rise to the following value reduction principle.

\begin{lem}
\label{lem:lowerred}
Let $\lambda>0$, $z\in\mathbb{C}^{d\times d}_{\rm s.a.}$ 
with
$z^{-1}>\frac{1}{\lambda}\frac{a_0a_0^*}{n+1}$,
$\sum_{j=1}^n N_j^*zN_j<\lambda\id$.
Define
$$
	G(\lambda;z) =
	\begin{bmatrix}
	\lambda^{-1}\id & 0 & 0 \\
	0 & \big(\lambda\id - \sum_{j=1}^n N_j^*zN_j\big)^{-1} & 0 \\
	0 & 0 & z
	\end{bmatrix}.
$$
Then $G(\lambda;z)>0$, $G(\lambda,z)^{-1}>A_0$, and
$$
	F(G(\lambda;z)) =
	\begin{bmatrix}
	\lambda\id & 0 & 0 \\
	0 & \lambda\id & 0 \\
	0 & 0 & z^{-1} + \sum_{i=1}^n
	M_i\big(\lambda\id- \sum_{j=1}^n N_j^*zN_j\big)^{-1} M_i^*
	\end{bmatrix}.
$$
In particular, if $Z$ is as in Lemma \ref{lem:lowerschur}, then
$F(G(\lambda;z)) \le F(Z)$.
\end{lem}

\begin{proof}
The only conclusion of the lemma that does not follow immediately from the 
definitions is $G(\lambda,z)^{-1}>A_0$, that is, that
$$
	\begin{bmatrix}
	\lambda\id & 0 & -\frac{a_0^*}{\sqrt{n+1}} \\
	0 & \lambda\id - \sum_{j=1}^n N_j^*zN_j & 0 \\
	-\frac{a_0}{\sqrt{n+1}} & 0 & z^{-1}
	\end{bmatrix}>0.	
$$
By Lemma \ref{lem:schur}, this is equivalent to
$$
	\lambda\id 
	-
	\frac{1}{n+1}
	\begin{bmatrix} 0 & a_0^* \end{bmatrix}
	\begin{bmatrix}\lambda\id - \sum_{j=1}^n N_j^*zN_j & 0
	\\
	0 & z^{-1}
	\end{bmatrix}^{-1}
	\begin{bmatrix} 0 \\ a_0 \end{bmatrix}>0,
$$
which is further equivalent to the assumption that 
$z^{-1}>\frac{1}{\lambda}
\frac{a_0a_0^*}{n+1}$ by the argument given at the end of the proof of
Lemma \ref{lem:lowerschur}.
\end{proof}

We readily deduce the following.

\begin{cor}
\label{cor:lowerprevar}
For $x$ as in \eqref{eq:mainx}, we have
$$
	-\lambda_{\rm min}(xx^*+b\otimes\id) = 
	\inf_{\substack{
	z^{-1}>\frac{a_0a_0^*}{n+1} \\
	\sum_j \! N_j^*zN_j < \id}}
	\lambda_{\rm max}\Bigg(
	z^{-1} + \sum_{i=1}^n
        M_i\Bigg(\id- \sum_{j=1}^n N_j^*zN_j\Bigg)^{-1} M_i^*	
	\Bigg).
$$
\end{cor}

\begin{proof}
Apply Lemmas \ref{lem:lowerlehner} and \ref{lem:lowerred}
exactly as in the proof of Proposition \ref{prop:mainupper}.
\end{proof}

Unlike in \S\ref{sec:upper}, the complication that now arises is that 
Corollary~\ref{cor:lowerprevar} depends only implicitly on the matrices 
$a_1,\ldots,a_n,b$ that define the operator on the left-hand side. We must 
therefore further simplify the variational principle to arrive at
the explicit form that is stated in Theorem \ref{thm:main}.
To this end, define the matrices
$$
	N=\begin{bmatrix} N_1 \\ \vdots \\ N_n \end{bmatrix},
	\qquad\qquad
	A=\begin{bmatrix} a_1 \\ \vdots \\ a_n \end{bmatrix}
$$
in $\mathbb{C}^{n\times 1}\otimes\mathbb{C}^{d\times 2nd}$ and
$\mathbb{C}^{n\times 1}\otimes\mathbb{C}^{d\times m}$, respectively.
Then we have the following.

\begin{lem}
\label{lem:lowerwoodbury}
Let $z^{-1}>\frac{a_0a_0^*}{n+1}$. Then $\sum_j N_j^*zN_j < \id$ and
\begin{multline*}
	\sum_{i=1}^n M_i
	\Bigg(\id- \sum_{j=1}^n N_j^*zN_j\Bigg)^{-1} M_i^* 
	=  \\
	\frac{n}{n+1}a_0a_0^*
	-
	\sum_{i=1}^n a_i 
	\Bigg(\id + 
	\sum_{j=1}^n a_j^*
	\bigg( z^{-1} -
        \frac{a_0a_0^*}{n+1}\bigg)^{-1}a_j \Bigg)^{-1}
	a_i^*.
\end{multline*}
\end{lem}

\begin{proof}
We first note that $M_iN_j^* = - a_ia_j^*$
and $M_iM_j^* = N_iN_j^* = 1_{i=j} \frac{a_0a_0^*}{n+1} - a_ia_j^*$
by the definition of $M_i,N_i$.
We can therefore write
$$
	\id\otimes z^{-1}-NN^* =
	\id\otimes\bigg(z^{-1}-\frac{a_0a_0^*}{n+1}\bigg) + AA^* > 0,
$$
which yields the desired inequality
\begin{align*}
	\sum_{j=1}^n N_j^*zN_j =
	N^*(\id\otimes z)N &= 
	N^*\big(\id\otimes z^{-1}-NN^*+NN^*\big)^{-1}N  \\
	&=
	\id - \big(\id + N^*
	\big(\id\otimes z^{-1}-NN^*\big)^{-1} N \big)^{-1}
	< \id
\end{align*}
by Lemma \ref{lem:woodbury}.
Next, we rearrange the
previous equation display as
$$
	\Bigg(\id- \sum_{j=1}^n N_j^*zN_j\Bigg)^{-1} 
	= \id + 
	N^*\big(
	\id\otimes z^{-1} - NN^*
	\big)^{-1}N.
$$
As $M_iN^* = a_i A^*$ and $M_iM_i^*=\frac{a_0a_0^*}{n+1} - a_ia_i^*$,
we can compute
\begin{multline*}
	\sum_{i=1}^n M_i
	\Bigg(\id- \sum_{j=1}^n N_j^*zN_j\Bigg)^{-1} M_i^* 
	=  \\
	\frac{n}{n+1}a_0a_0^*
	+ 
	\sum_{i=1}^n a_i 
	\Bigg[
	A^*
	\Bigg(
	\id\otimes\bigg( z^{-1} - 
	\frac{a_0a_0^*}{n+1}\bigg) + AA^*
	\Bigg)^{-1}A - \id \Bigg] a_i^*.
\end{multline*}
Applying Lemma \ref{lem:woodbury} again yields the conclusion.
\end{proof}

We can now prove the second part of Theorem \ref{thm:main}.

\begin{prop}
\label{prop:mainlower}
For $x$ as in \eqref{eq:mainx}, we have
$$
	\lambda_{\rm min}(xx^*+b\otimes\id) = 
	\sup_{z<0}\,
	\lambda_{\rm min}\Bigg(
	b + z^{-1} +
	\sum_{i=1}^n a_i 
	\Bigg(\id - 
	\sum_{j=1}^n a_j^*za_j \Bigg)^{-1}
	a_i^*
	\Bigg).
$$
\end{prop}

\begin{proof}
This follows readily from Corollary \ref{cor:lowerprevar},
Lemma \ref{lem:lowerwoodbury}, and $a_0a_0^*=-b$, where
we make the substitution $z^{-1} - \frac{a_0a_0^*}{n+1}
\leftarrow -z^{-1}$.
\end{proof}

\section{The matrix Cauchy transform}
\label{sec:cauchy}

It remains to prove the last part of Theorem \ref{thm:main}, viz., that 
the variational principles can be restricted to $z$ such that the matrix 
on the right-hand side is a multiple of the identity. This will be deduced 
from another result that is useful in its own right. Lemma 
\ref{lem:reduction} will follow as a byproduct of the proof.

Define an analytic function
$$
	G:\mathbb{C}\backslash\spc(xx^*+b\otimes\id) \to
	\mathbb{C}^{d\times d},
$$
the \emph{matrix Cauchy transform} of $xx^*+b\otimes\id$, by
$$
	G(\lambda) := 
	({\mathrm{id}\otimes\tau})\Big[
	\big(\lambda - xx^* - b\otimes\id\big)^{-1}
	\Big].
$$
Here $\tau$ is the trace on the $C^*$-probability space 
$(\mathcal{A},\tau)$ in which the 
free semicircular family $s_1,\ldots,s_n$ is defined. The following
is the main result of this section.

\begin{thm}
\label{thm:cauchy}
For every $\lambda\in \mathbb{C}\backslash
\mathrm{conv}(\spc(xx^*+b\otimes\id))$, we have
$$
	b + G(\lambda)^{-1} +
	\sum_{i=1}^n a_i \Bigg(
	\id - \sum_{j=1}^n a_j^*G(\lambda) a_j
	\Bigg)^{-1} a_i^* = \lambda\id.
$$
In particular, both inverses in this equation exist.
\end{thm}

In the case that $d=1$ (that is, when the coefficients $a_i,b$ are 
scalar), the equation in Theorem \ref{thm:cauchy} reduces to the well 
known equation for the Cauchy 
transform of the free Poisson distribution \cite[pp.\ 203--206]{NS06}.
This explains the assertion made in the introduction that the operators
considered in this note may be viewed as matrix-valued analogues of the 
free Poisson distribution.

\begin{proof}[Proof of Theorem \ref{thm:cauchy}]
When $|\lambda|$ is sufficiently large that
$g(\lambda)=(\lambda-b\otimes\id)^{-1}$ satisfies
$\|g(\lambda)\|\|xx^*\|<1$, $G(\lambda)$ has a convergent power 
series expansion
$$
	G(\lambda) =
	\sum_{k=0}^\infty
	({\mathrm{id}\otimes\tau})\Big[
	g(\lambda)\big(xx^*g(\lambda)\big)^k\Big].
$$
Let us fix such a $\lambda$ until further notice, and define
\begin{align*}
	C_k &=
	({\mathrm{id}\otimes\tau})\Big[
	\big(xx^*g(\lambda)\big)^k\Big],
\\
	D_k &=
	({\mathrm{id}\otimes\tau})\Big[ 
	x^*g(\lambda) 
	\big(xx^*g(\lambda)\big)^k x\Big].
\end{align*}
Reasoning as in the proof of \cite[Lemma 4.4]{vH25cdm} and
using that the trace of a product of an odd number of 
semicircular variables vanishes, we obtain the recursions
\begin{align*}
	C_{k+1} &=
	\sum_{i=1}^n a_ia_i^*(\lambda-b)^{-1} C_k
	+
	\sum_{l=0}^{k-1} \sum_{i=1}^n a_i D_l
	a_i^* (\lambda-b)^{-1} C_{k-1-l}, \\
	D_k &=
	\sum_{i=1}^n a_i^*(\lambda-b)^{-1} C_k a_i +
	\sum_{l=0}^{k-1}
	\sum_{i=1}^n a_i^*(\lambda-b)^{-1}C_l a_i
	D_{k-1-l}
\end{align*}
for $k\ge 0$ with initial condition $C_0=\id$.
Summing these recursions over $k$ yields
\begin{align}
\label{eq:GGG}
	(\lambda-b)
	G(\lambda) &= \id +
	\sum_{i=1}^n a_i 
	H(\lambda)
        a_i^* 
	G(\lambda), \\
\label{eq:HHH}
	H(\lambda) &=
	\id+
        \sum_{i=1}^n a_i^* G(\lambda) a_i  
	H(\lambda),
\end{align}
where we define
$$
	H(\lambda) := \id+\sum_{k=0}^\infty D_k =
	\id +
	({\mathrm{id}\otimes\tau})\Big[
	x^*\big(\lambda - xx^* - b\otimes\id\big)^{-1}x
	\Big].
$$
We have therefore established the validity of 
\eqref{eq:GGG}--\eqref{eq:HHH} for all $\lambda\in\mathbb{C}$ with
$|\lambda|$ sufficiently large. However, since both $G$ and $H$ are 
analytic, the validity of these equations extends to every
$\lambda\in \mathbb{C}\backslash\spc(xx^*+b\otimes\id)$.

Now note that as $\mathrm{Im}\,\frac{1}{\lambda-t} = 
-\frac{\mathrm{Im}\,\lambda}{|\lambda-t|^2}$ for $t\in\mathbb{R}$, it 
follows from the definitions of $G(\lambda)$ and $H(\lambda)$ that 
$\mathrm{Im}\,G(\lambda)$ and 
$\mathrm{Im}\,H(\lambda)$ are 
negative definite when $\mathrm{Im}\,\lambda>0$ and positive definite when 
$\mathrm{Im}\,\lambda<0$. Thus $G(\lambda)$ and $H(\lambda)$ are 
invertible whenever $\mathrm{Im}\,\lambda\ne 0$ by \cite[Lemma~3.1]{HT05}. 
Consequently, \eqref{eq:HHH} yields
\begin{equation}
\label{eq:Hdef}
	H(\lambda) = \Bigg(
	\id - \sum_{j=1}^n a_j^* G(\lambda) a_j
	\Bigg)^{-1},
\end{equation}
and substituting this expression into \eqref{eq:GGG} and multiplying on 
the right by $G(\lambda)^{-1}$ yields the conclusion
in the case that $\mathrm{Im}\,\lambda\ne 0$.

It remains to consider the cases where $\lambda$ is real.
If $\lambda>\lambda_{\rm max}(xx^*+b\otimes\id)$, then
$G(\lambda)$ and $H(\lambda)$ are
positive definite and hence invertible by their definition, and
the proof is concluded as above.
If $\lambda<\lambda_{\rm min}(xx^*+b\otimes\id)$, then $G(\lambda)$ is 
negative definite and thus $\id-\sum_j a_j^*G(\lambda)a_j$
is positive definite. Together with \eqref{eq:HHH}, this implies
that $G(\lambda)$ and $H(\lambda)$ are invertible, and we again
conclude as above.
\end{proof}

We can now conclude the proof of Theorem \ref{thm:main}.

\begin{proof}[Proof of Theorem \ref{thm:main}]
The two variational principles were proved above in
Propositions~\ref{prop:mainupper}~and~\ref{prop:mainlower}, respectively.
It is convenient to write
$$
	h(z) = 
        b + z^{-1} +
        \sum_{i=1}^n a_i \Bigg(   
        \id - \sum_{j=1}^n a_j^*z a_j
        \Bigg)^{-1} a_i^*,
$$
so that the two variational principles may by expressed as
\begin{align*}
	\lambda_{\rm max}(xx^* + b\otimes\id) &=
        \inf_{\substack{z>0 \\ \sum_j \! a_j^*za_j<\id}}
	\lambda_{\rm max}(h(z)), \\
	\lambda_{\rm min}(xx^* + b\otimes\id) &=
        \sup_{z<0}\,
	\lambda_{\rm min}(h(z)).
\end{align*}
It remains to prove the last assertion of the theorem.

Let us first consider the second variational principle. Clearly
$$
	\lambda_{\rm min}(xx^* + b\otimes\id) =
	\sup_{z<0}\,
        \lambda_{\rm min}(h(z))
	\ge 
	\sup_{\lambda\in\mathbb{R}}
	\sup_{\substack{z<0 \\ h(z)=\lambda\id }}
	\,\lambda_{\rm min}(h(z)),
$$
since we restrict the supremum to a smaller set.
On the other hand, fix any $\varepsilon>0$ and let
$\mu = \lambda_{\rm min}(xx^* +b\otimes\id) - \varepsilon$.
Then $G(\mu)<0$ and $h(G(\mu))=\mu\id$ 
by the 
definition of the matrix Cauchy transform and Theorem \ref{thm:cauchy}, 
respectively. Therefore
$$
	\sup_{\lambda\in\mathbb{R}}
	\sup_{\substack{z<0 \\ h(z)=\lambda\id }}
	\,\lambda_{\rm min}(h(z))
	\ge
		\lambda_{\rm min}(h(G(\mu))) = 
	\lambda_{\rm min}(xx^* +b\otimes\id) - \varepsilon.
$$
Letting $\varepsilon\downarrow 0$ shows that 
$$
	\lambda_{\rm min}(xx^* +b\otimes\id) =
	\sup_{\lambda\in\mathbb{R}}
	\sup_{\substack{z<0 \\ h(z)=\lambda\id }}
        \,\lambda_{\rm min}(h(z)),
$$
which is the desired conclusion.

For the first variational principle, let $\varepsilon>0$ and $\mu = 
\lambda_{\rm max}(xx^* +b\otimes\id) + \varepsilon$. Then $G(\mu)>0$ and 
$h(G(\mu))=\mu\id$ by the definition of the matrix Cauchy transform and 
Theorem \ref{thm:cauchy}. Moreover, that $\sum_j a_j^* G(\mu) a_j <\id$ 
follows from \eqref{eq:Hdef} since $H(\mu)>0$ by its definition. The proof 
for the first variational principle can now be completed in the identical 
manner as for the second variational principle.
\end{proof}

We conclude by proving Lemma \ref{lem:reduction}. A similar result in the 
setting of Theorem~\ref{thm:lehner} appears in \cite[Lemma~7.1]{BCSV24}, 
but the proof given here is entirely different.

\begin{proof}[Proof of Lemma \ref{lem:reduction}]
Under the assumptions of Lemma \ref{lem:reduction},
it is readily verified from the recursions in the proof
of Theorem \ref{thm:cauchy} that $C_k\in\mathcal{B}$ and
$D_k\in\mathcal{A}$ for all $k\ge 0$. Thus $G(\lambda)\in \mathcal{B}$
and $H(\lambda)\in\mathcal{A}$ for all 
$\lambda\in\mathbb{C}\backslash\spc(xx^*+b\otimes\id)$. The rest of the 
proof now proceeds exactly as in the proof of the last part of Theorem 
\ref{thm:main}.
\end{proof}


\subsection*{Acknowledgments}

RvH was supported in part by NSF grant DMS-2347954.

\bibliographystyle{abbrv}
\bibliography{ref}

\end{document}